\documentclass[12pt]{article}
\usepackage[psamsfonts]{amssymb}
\usepackage{amsmath}
\usepackage{amsfonts,euscript,amssymb,amsmath,amsthm}
\usepackage{graphicx}
\usepackage{mathrsfs}

\counterwithin{equation}{section}

\newtheorem{theorem*}{Theorem}
\newtheorem{lemma}{Lemma}
\newtheorem{corollary}{Corollary}
\newtheorem{proposition}{Proposition}

\begin{document}

\begin{center}
{\bf \large  On the constant in Klartag's theorem for lattice packings in high dimensions.}
\end{center}

\begin{center}

{\bf I.~Rezvyakova}
\let\thefootnote\relax


{\small  Steklov Mathematical Institute of the Russian Academy of Sciences, Moscow\\
HSE University, Moscow 
}
\end{center}

\begin{abstract}
 {\small  In 2025, Boaz Klartag proved that there exists a lattice sphere packing in
$\mathbb{R}^n$ whose density is at least
\[
    c n^2 2^{-n},
\]
where $c>0$ is an absolute constant. Our aim is to refine the analysis of
Klartag's argument and to show that the constant can be taken as
\[
    c = \frac{1}{2e} - o(1),
    \qquad n\to\infty.
\]}
\end{abstract}
\hspace{1mm}

\begin{center}
{\bf \S 1. Technical lemma. }
\end{center}

\begin{lemma}
Let $n$ be a large integer, and let $\alpha \ge1, C >1, T>0$ be parameters, possibly depending on $n$. Define 
\[
    J(t)=\int_0^{C}\Phi(y)
    \left(1-\frac{y\sqrt t}{\alpha}\right)^{-(n+2)/2}\,dy,
\]
where  
\begin{equation}\label{eq0}
    \Phi(y)\le \min\left\{\frac12,\frac{e^{-y^2/2}}{\sqrt{2\pi}\,y}\right\}.
\end{equation}
If $\max (n^3T^2, n^{-1} T^{-1/2}, C\sqrt{T}) = o(1)$, then
\[
    \int_0^T \frac{n\sqrt t}{2}J(t)\,dt
    \le
    (1+o(1))\frac{8}{n^2}\exp\left(\frac{n^2T}{8}\right).
\]
\end{lemma}

\begin{proof}
Since $y \ge 0$, \(\alpha\ge1\), and $C\sqrt{T} = o(1)$, we shall use the estimate 
\[
    \left(1-\frac{y\sqrt t}{\alpha}\right)^{-(n+2)/2}
    \le
    (1-y\sqrt t)^{-(n+2)/2} .
\]
Thus
\[
    J(t)\le
    \int_0^{C}\Phi(y)
    (1-y\sqrt t)^{-(n+2)/2}\,dy.
\]
We split
\[
    J(t)=J_1(t)+J_2(t),
\]
where \(J_1\) corresponds to the range \(0\le y\le1\), and \(J_2\) to the range 
\(1\le y\le C\).

For \(0\le y\le1\), using the bound \(\Phi(y)\le1/2\), we obtain
\[
    J_1(t)\le
    \frac12(1-\sqrt t)^{-(n+2)/2}.
\]
Since \( \max (n T, \sqrt T) = o(1)\),
\[
    (1-\sqrt t)^{-(n+2)/2}  \le
    \exp\left(\frac{n\sqrt t}{2}+O(\sqrt T + n T)\right)
    =
    \exp\left(\frac{n\sqrt t}{2}\right)\left(1 +o(1)\right).
\]
Therefore
\[
    \int_0^T \frac{n\sqrt t}{2}J_1(t)\,dt
    \ll
    T\exp\left(\frac{n\sqrt T}{2}\right) =
    o\left(\frac1{n^2}e^{n^2T/8}\right)
\]
since $n^2 T \to +\infty$.

For \(y\ge1\), we use
\(
    \Phi(y)\le \frac{e^{-y^2/2}}{\sqrt{2\pi}\,y}.
\)
Thus, 
\[
    J_2(t)\le
    \int_1^{C}
    \frac{e^{-y^2/2}}{\sqrt{2\pi}\,y}
    (1-y\sqrt t)^{-(n+2)/2}\,dy.
\]
Consequently,
\[
\begin{aligned}
\int_0^T \frac{n\sqrt t}{2}J_2(t)\,dt
&\le
\int_1^{C}
\frac{e^{-y^2/2}}{\sqrt{2\pi}\,y}
\int_0^T
\frac{n\sqrt t}{2} (1-y\sqrt t)^{-(n+2)/2}
\,dt\,dy.
\end{aligned}
\]
We first consider the range that gives the main contribution.  Put
\[
    a=\frac{n\sqrt T}{2}
\]
and consider the interval
\[
    \left|y- a\right|
    \le
    \sqrt{a}.
\]
In this range \(y=O(n\sqrt T)\), and, therefore, 
\[
    ny^2T=O(n^3T^2) = o(1).
\]
Thus, in the main range, uniformly for \(0\le t\le T\),
\[
    (1-y\sqrt t)^{-(n+2)/2}
    \le
    (1+o(1))\exp\left(\frac{ny\sqrt t}{2}\right).
\]
This part contributes, in total, 
\[
\begin{aligned}
(1+o(1)) \frac{n\sqrt T}{2}\int\limits_{|y-a| \le \sqrt{a}} 
\frac{e^{-y^2/2}}{\sqrt{2\pi}\,y}
\int_0^T
 \exp\left(\frac{ny\sqrt t}{2}\right)
\,dt\,dy \le (1+o(1))  \frac{n\sqrt T}{2} \int\limits_{|y-a| \le \sqrt{a}} 
\frac{e^{-y^2/2}}{\sqrt{2\pi}\,y} \frac{4 \sqrt{T}}{ny} e^{ny\sqrt{T}/2} \,dy,
\end{aligned}
\]
since 
\begin{equation}\label{eq1}
\int_0^T
 \exp\left(\frac{ny\sqrt t}{2}\right)
\,dt\ \le \frac{4 \sqrt{T}}{ny} e^{ny\sqrt{T}/2}.
\end{equation}
Note that 
$$
-\frac{y^2}{2} + \frac{ny \sqrt{T}}{2} = -\frac12 \left( y- a\right)^2 +\frac{ n^2 T}{8}.
$$
Therefore, the main range contributes
\[
\begin{aligned}
 (1+o(1))  \frac{8}{n^2}  e^{n^2 T/8}
\end{aligned}
\]
since $\int\limits_{-\infty}^{+\infty} e^{-y^2/2} dy = \sqrt{2\pi}$.

In the complementary range we use the bound
\[
    (1-y\sqrt t)^{-(n+2)/2}
    \le
    \exp\left(\frac{ny\sqrt t}{2} +  O(n y^2 T) + O(C\sqrt{T})\right), 
\]
together with (\ref{eq1}) and $y^{-2} \le 1$ for $y \ge 1$. Since $C\sqrt{T} = o(1)$, the contribution of the remaining range is boundefd by 
\[
\begin{aligned}
& (1+o(1)) 2T  \int\limits_{|y-a| \ge \sqrt{a}} 
\frac{1}{\sqrt{2\pi}}
\exp\left(-\frac{y^2}{2}(1- O(n T))  + \frac{ny\sqrt T}{2}\right)\,dy \\
&\ll  T e^{a^2 (1+O(a\sqrt{T}))/2} \int\limits_{|y-a| \ge \sqrt{a}} \exp\left(-\frac{1 - O(a \sqrt{T})}{2} \left( y- \frac{a}{1 - O(a \sqrt{T})}\right)^2\right)\,dy \\
&\ll T e^{n^2 T/8} e^{O(a^3 \sqrt{T})}\int\limits_{|y| \ge \sqrt{a}} \exp\left(-\frac{1 - O(a\sqrt{T})}{2} (y -  O(a^2 \sqrt{T}))^2\right)\,dy.
\end{aligned}
\]
Since 
\[
a^2 \sqrt{T} \ll n^2 T^{5/2} = o(1), \, a^{-1} \ll n^{-1} T^{-1/2} = o(1), 
\]
we may drop the shift $O(a^2 \sqrt{T})$ when $n$ is sufficiently large by widening the range of integration to $|y| \ge \sqrt{a}-1$. Since $a^3 \sqrt{T} \ll n^3 T^{5/2} \ll n^3 T^2 = o(1)$, the term $e^{O(a^3 \sqrt{T})}$ is bounded by an absolute constant. Also, since $a \sqrt{T} \ll  nT = o(1)$, we may replace the coefficient $\frac{1 - O(a\sqrt{T})}{2}$ by $1/3$. 
Thus,  up to some absolute constant , 
the last expression is bounded by 
$$
T e^{n^2 T/8}  \int\limits_{|y| \ge \sqrt{a}-1} \exp\left(-\frac{y^2}{3} \right) dy \ll T e^{n^2 T/8} e^{- a/4} = n^{-2} e^{n^2 T/8} (T n^2 e^{-n \sqrt{T}/8}) = o(n^{-2} e^{n^2 T/8}). 
$$
Combining all  the estimates completes the proof  of the lemma.
\end{proof}

\begin{center}
{\bf \S 2. Main lemmas and the main result. }
\end{center}

Throughout we shall use the following notation.
\begin{itemize}
\item $n$ is the dimension of the Euclidean space $\mathbb R^n$
\item $\kappa_n$ is the volume of the unit ball in $\mathbb R^n$
\item $ \Phi(y)$ is defined by (\ref{eq0})
\item $\mathscr{X}_n$ is the space of all lattices $L$ in $\mathbb R^n$ with $\operatorname{vol} \left( \mathbb R^n / L \right) = \kappa_n$, which is endowed with 
the unique Haar probability measure invariant under the action of $ SL_n (\mathbb R)$
\item For a lattice $L$, let $a_0  = a_0  (L)>0$ be a parameter such that 
the ball of radius $1/\sqrt{a_0}$ is $L$-free
\item For a fixed lattice \(L\), let \(\mathbb{E}_{A}\) denote the mathematical expectation with respect to the probability measure of the matrix-valued stochastic deformation process $(A_t)_{t \ge 0}$ introduced by Klartag, where $A_0 = a_0 I$ and the time parameter \(T > 0\) is deterministic. 
\item $\mathcal{E}_t$ is the ellipsoid corresponding to the matrix $A_t$, and $\partial \mathcal{E}_t$ is its boundary.
\item $R_t = R_t (\alpha) = \{ x \in \mathbb R^n :  \frac{1}{\alpha} \le |x|^2 < \frac{1}{\alpha - C_0\sqrt{tn}}\}$, where the constant $C_0$ is absolute and provided by Corollary 3.2 of \cite{K25}
\item  $ \widetilde K_t(L) = \widetilde K_t(L, \alpha) =  \sum\limits_{x \in L\cap R_t (\alpha)} 
\Phi\left( \frac{1}{\sqrt{t}} \left( \alpha-\frac{1}{|x|^2}\right) \right)$, with $\Phi (\cdot)$ as in (\ref{eq0})
\item $\Delta_n$ is the optimal lattice-packing in $\mathbb R^n$
\end{itemize}

Our first goal is  to prove the quantitative version of Proposition 5.1 of \cite{K25}, using the technical lemma from \S 1.
\begin{proposition}\label{prop1}
Let \(n\to\infty\),  \(1 \le \rho=\rho(n) =o(\sqrt n) \), and put
\[
    \alpha=\left( 1-\frac{\rho}{n} \right)^{-2}.
\]
Let
\[
    n^{-2} \log n \le T\le 20n^{-2}\log n.
\]
Then
\[
    \mathbb E_L\int_0^T \widetilde K_t(L)\,dt
    \le
    8(1+o(1)) e^{-\rho}
    \exp\left(\frac{n^2T}{8}\right)n^{-2}.
\]
\end{proposition}

\begin{proof}
We follow the proof of Proposition 5.1 in \cite{K25}, keeping track of the
dependence on \(\alpha\). By the definition of $\widetilde K_t(L)$ and by (63) of \cite{K25},
$$
\mathbb E_L  \widetilde K_t(L) = \mathbb E_L \sum\limits_{x \in L\cap R_t} 
\Phi\left( \frac{1}{\sqrt{t}} \left( \alpha -\frac{1}{|x|^2}\right) \right) = \kappa_n^{-1} \int\limits_{R_t} \Phi\left( \frac{1}{\sqrt{t}} \left( \alpha -\frac{1}{|x|^2}\right) \right) dx := I (t).
$$
By the 
non-negativity of $\widetilde K_t(L)$,   Fubini's theorem implies 
$$
\mathbb E_L  \int\limits_0^T \widetilde K_t(L) dt =  \int\limits_0^T \mathbb E_L \widetilde K_t(L) dt = 
 \int\limits_0^T I(t) dt.
$$
As in the proof of Lemma 4.3 of \cite{K25}, we arrive at the expression 
$$
I(t)  =  \frac{n\sqrt t}{2} \alpha^{-(n+2)/2} J(t),
$$
where 
\[
    J(t)
    =
    \int_0^{C_0\sqrt n}
    \Phi(y)
    \left(1-\frac{y\sqrt t}{\alpha}\right)^{-(n+2)/2}\,dy
\]
and $C_0$ is some absolute constant.
We now apply Lemma~1 with
\[
    C=C_0\sqrt n.
\]
Note that all the assumptions of Lemma~1 are satisfied. 
Hence
\[
    \int_0^T
    \frac{n\sqrt t}{2}J(t)\,dt
    \le
    (1+o(1))
    \frac8{n^2}
    \exp\left(\frac{n^2T}{8}\right).
\]
Consequently,
\[
    \int_0^T
   I(t)\,dt
    \le
    (1+o(1))
    \alpha^{-(n+2)/2}
    \frac8{n^2}
    \exp\left(\frac{n^2T}{8}\right).
\]
Recall that
\[
    \alpha=\left(1-\frac{\rho}{n}\right)^{-2}.
\]
Thus
\[
    \alpha^{-(n+2)/2}
    =
    \exp\left(
        (n+2)\log\left(1-\frac{\rho}{n}\right)
    \right).
\]
Since \(\rho=o(\sqrt n)\), we have
\[
    (n+2)\log\left(1-\frac{\rho}{n}\right)
    =
    -\rho+O\left(\frac{\rho^2}{n}\right)+O\left(\frac{\rho}{n}\right)
    =
    -\rho+o(1)
\]
and, therefore, 
\[
    \alpha^{-(n+2)/2}
    =
    e^{-\rho}(1+o(1)).
\]
Substituting this into the preceding estimate finishes the proof of the proposition.
\end{proof}

\begin{corollary}
Under the assumptions of Proposition \ref{prop1}, for any  \(p=p(n)\) satisfying 
\[
    0<p<1-e^{-\rho},
\]
there exists a lattice \(L \in \mathscr{X}_n\) such that $(1-\rho/n)^{-2} |x|^2 \ge 1$ for every non-zero $x\in L$ and  
\[
    \int_0^T \widetilde K_t(L)\,dt
    \le
    \frac{8}{p}(1+o(1))e^{-\rho} \exp\left(\frac{n^2T}{8}\right)n^{-2},
\]
or, equivalently, in the notation (50), (51) of \cite{K25} with $a_0 = (1-\rho/n)^{-2}$, 
\[
    \int_0^T  K_t(L)\,dt
    \le
    \frac{8}{p}(1+o(1))e^{-\rho} \exp\left(\frac{n^2T}{8}\right)n^{-2}.
\]
\end{corollary}

\begin{proof}
In Klartag's equality (63) take $\varphi(x) = 1$ if $|x| \le 1- \frac{\rho}{n}$ and $\varphi(x) = 0$ otherwise. We then have 
$$
\mathbb E_L \sum\limits_{x \in L \setminus \{0\}} \varphi(x) = \kappa_n^{-1} \int\limits_{\mathbb R^n} \varphi(x)  dx =  \left( 1- \frac{\rho}{n} \right)^n \le e^{-\rho}.
$$
Hence, by the Markov inequality, for random $L \in \mathscr{X}_n$ 
$$
\mathbb P (\exists \, x\in L \setminus \{0\}, |x| \le 1- \frac{\rho}{n}) \le e^{-\rho}.
$$

By Proposition \ref{prop1} and the Markov inequality, 
\[
    \mathbb P\left(
        \int_0^T \widetilde K_t(L)\,dt >
        \frac{8(1+o(1))}{p}e^{-\rho} \exp \left(\frac{n^2 T}{8} \right)  n^{-2}
    \right)
    \le p.
\]
Since $p + e^{-\rho} <1$, we conclude that there exists a lattice $L \in \mathscr{X}_n$ such that $|x| > 1- \frac{\rho}{n}$ for every non-zero $x \in L$
 and 
\[
    \int_0^T \widetilde K_t(L)\,dt
    \le
    \frac{8(1+o(1))}{p}e^{-\rho} \exp \left(\frac{n^2 T}{8} \right)  n^{-2}.
\]
Therefore, for this  $L$ we have $\widetilde K_t(L) = K_t(L)$, which proves the corollary. 
\end{proof}

\begin{proposition}[Quantitative form of Proposition 3.4]
Let \(n\) be sufficiently large. Assume that
\[
    0<T\le n^{-5/3},
\]
and suppose the parameter $a_0$ for a fixed lattice $L$ satisfy 
\[
    1\le a_0\le 1+\frac{\rho}{n},
    \qquad
    1\le \rho\le n^{2/3}.
\]
Then, in the notation of Klartag's deformation process,
\[
\begin{aligned}
    \mathbb E_A \log\det A_T
    \le\;&
    n\log a_0
    -\frac{n^2 T}{4}
    +\frac14\int_0^T
        \mathbb E_A\,|\partial \mathcal{E}_t\cap L|\,dt       \\
    &\quad
    +O\!\left(n \rho T+n^{5/2}T^{3/2}\right).
\end{aligned}
\]
The constant in the \(O\)-term is absolute.
\end{proposition}

\begin{proof}
In the notation of  \cite{K25}, Lemma 3.3 there states that
\[
\begin{aligned}
    \mathbb E_A\log\det A_T
    \le\;&
    n\log a_0
    -\frac12\int_0^T
        \mathbb E_A\,[\|A_{t}\|^{-2}_{op} \cdot N_t ]\,dt.
\end{aligned}
\]
In \cite{K25} the first term was absorbed into an absolute constant ,
but here we keep it explicit, which is essential for a sharper bound. After using (43) of \cite{K25}, we have 
\[
\begin{aligned}
\mathbb E_A\,[\|A_{t}\|^{-2}_{op} \cdot N_t ]\, \ge   \mathbb E_A\,[N_t ] - C\sqrt{t} \,n^{5/2} - C \rho \,n.
\end{aligned}
\]
Integrating over $t$,  we obtain
\[
\begin{aligned}
\int_0^T \mathbb E_A\,[\|A_{t}\|^{-2}_{op} \cdot N_t ]\, dt \ge   \int_0^T \mathbb E_A\,[N_t ]\, dt - O(T^{3/2}\,n^{5/2} + T \rho \,n).
\end{aligned}
\]
Then arguing as in the end of the proof of Proposition 3.4 in  \cite{K25}, we arrive at the desired statement, which  
completes the proof.
\end{proof}

\begin{theorem*}
As \(n\to\infty\), one has
\[
    \Delta_n
    \ge
    \left(\frac1{2e}-o(1)\right)n^2 2^{-n}.
\]
\end{theorem*}

\begin{proof}
We use 
Corollary 1 with a slowly increasing
parameter \(\rho=\rho(n)\). Let, for instance,
\[
    \rho=\log\log n
\]
and take 
\[
    T=\frac{16\log n+\gamma}{n^2},
\]
where \(\gamma \in [0, 4\log n]\) will be chosen later. 
For 
\[
    p= 1-2e^{-\rho}
\]
(which is $1-o(1)$ as $n \to +\infty$) there exists a lattice \(L\) such that the initial  ball of radius $1/\sqrt{a_0}$ is L-free and 
\[
    \int_0^T K_t(L)\,dt
    \le
    \frac{8}{p}(1+o(1))e^{-\rho}e^{\gamma/8} = 8 (1+o(1))e^{-\rho}e^{\gamma/8}.
\]
Therefore, by Proposition 2 and Proposition 4.2 of \cite{K25}, we find that for the lattice $L$
\[
\begin{aligned}
    \mathbb E_A \log\det A_T
    \le\;&
- 4 \log n -\frac{\gamma}{4} + 2\rho + 4 (1+o(1))e^{-\rho}e^{\gamma/8} + o(1).
\end{aligned}
\]
Hence, with positive probability in the stochastic deformation process,
\[
\begin{aligned}
    \log\det A_T
    \le\;&
- 4 \log n -\frac{\gamma}{4} + 2\rho + 4 (1+o(1))e^{-\rho}e^{\gamma/8} + o(1),
\end{aligned}
\]
and, by Proposition 2.3 of \cite{K25}, the matrix $A_T$ is (almost surely whenever this event occurs)  positive definite and  the corresponding ellipsoid is $L$-free.  
Take 
\[
    \gamma=8\rho + 8\log\frac{1}{2}.
\]
Then there exists an  $L$-free  matrix $A_T$ such that 
\[
\begin{aligned}
    \log\det A_T
    \le\;&
- 4 \log n + C
\end{aligned}
\]
with 
\[
\begin{aligned}
C \le 2 \log (2e)+ o(1).
\end{aligned}
\]
The resulting ellipsoid
\[
    \mathcal{E}_T=\{x\in\mathbb R^n:\langle A_Tx,x\rangle\le1\}
\]
contains no non-zero points of the lattice \(L\) and its volume is
\[
    \operatorname{vol}(\mathcal{E}_T)
    =
    \kappa_n(\det A_T)^{-1/2}.
\]
To obtain a lower bound for the packing density
in dimension $n$, we make a change of variables by 
applying a linear map $A_T^{1/2}$. In the new variables,  the ellipsoid becomes a unit ball centered at the origin,  
and the lattice $L$ transforms into one having no non-zero points inside this unit ball. 
The  covolume of the new lattice  is
$$
\kappa_n \cdot \det A_T^{1/2} = \kappa_n \cdot \sqrt{\det A_T} \le  \kappa_n n^{-2} 2e (1 + o(1)).
$$
Since, by construction, we may place a ball of radius $1/2$ at each point of the lattice without overlaps, the density of this lattice packing is 
the ratio of the volume of a ball of radius \(1/2\) to the covolume of the lattice,  and it therefore 
satisfies 
\[
    \Delta_n \ge  \left(\frac1{2e}-o(1)\right) n^2 2^{-n}.
\]
This completes the proof of the theorem.
\end{proof}

\end{document}